\documentclass[12pt,reqno]{amsart}

\usepackage{amssymb}
\usepackage{amscd}
\usepackage{amsfonts}
\usepackage{setspace}
\usepackage{version}
\usepackage{mathrsfs}

\usepackage{graphicx}

\newtheorem{theorem}{Theorem}[section]
\newtheorem{lemma}[theorem]{Lemma}
\newtheorem{proposition}[theorem]{Proposition}
\newtheorem{corollary}[theorem]{Corollary}
\newtheorem{conjecture}{Conjecture}

\theoremstyle{definition}

\newtheorem*{usp-dash}{Conjecture \ref{usp}'}
\newtheorem*{si-conj-star}{Conjecture \textbf{(SI*)}}
\newtheorem*{ulp-conj}{Conjecture \textbf{(ULP)}}
\newtheorem*{usp-conj}{Conjecture \textbf{(USP)}}
\newtheorem*{usp-star}{Conjecture \textbf{(USP*)}}
\newtheorem*{ud-conj}{Conjecture \textbf{(UD)}}
\newtheorem*{ak-conj}{Conjecture \textbf{(AK)}}
\newtheorem*{es-conj}{Conjecture \textbf{(ES)}}
\newtheorem*{ehm-conj}{Conjecture \textbf{(EHM)}}
\newtheorem*{eak-conj}{Conjecture \textbf{(EAK)}}
\newtheorem*{ff-conj}{Conjecture \textbf{(FF)}}

\renewcommand{\leq}{\leqslant}
\renewcommand{\geq}{\geqslant}

\def\F{\mathbf{F}}
\def\R{\mathbf{R}}

\def\Z{\mathbf{Z}}

\def\P{\mathbf{P}}
\def\Q{\mathbf{Q}}
\def\N{\mathbf{N}}
\def\X{\mathsf{X}}
\def\Y{\mathsf{Y}}
\def\bH{\mathbf{H}}
\def\eps{\varepsilon}

\newcommand{\md}[1]{\ensuremath{(\operatorname{mod}\, #1)}}

\usepackage{rotating}

\numberwithin{equation}{section}

\begin{document}

\title[Arithmetic Kakeya]{On the arithmetic Kakeya conjecture of Katz and Tao}


\author{Ben Green}
\address{Mathematical Institute, Radcliffe Observatory Quarter, Woodstock Rd, Oxford OX2 6GG}
\email{ben.green@maths.ox.ac.uk}

\author{Imre Z.~Ruzsa}
\address{Alfred R\'enyi Institute, Budapest, Re\'altanoda utca 13-15., 1053 Hungary}
\email{ruzsa@renyi.hu}

\begin{abstract}
The arithmetic Kakeya conjecture, formulated by Katz and Tao in 2002, is a statement about addition of finite sets. It is known to imply a form of the Kakeya conjecture, namely that the upper Minkowski dimension of a Besicovitch set in $\R^n$ is $n$. In this note we discuss this conjecture, giving a number of equivalent forms of it. We show that a natural finite field variant of it does hold. We also give some lower bounds.
\end{abstract}
\maketitle

\tableofcontents

\section{Introduction and statement of results}

The arithmetic Kakeya conjecture, sometimes known as the sums-differences conjecture, was formulated by Katz and Tao around fifteen years ago. It is a purely additive-combinatorial statement which, if true, would have a deep geometric consequence -- that the Minkowski dimension of Besicovitch sets in $\R^n$ is $n$. This is the celebrated Kakeya conjecture, discussed at length in many places: for an introduction see \cite{tao-needle}. 

The arithmetic Kakeya conjecture is mentioned explicitly\footnote{In the earlier paper \cite[p. 234]{katz-tao-2} of Katz and Tao, the authors only go so far as to suggest that it is ``not too outrageous tentatively to conjecture'' this statement. In fact, the conjecture made in \cite{tao-edinburgh} is over fields of ``sufficiently large characteristic'' (or characteristic zero) whereas this paper provides evidence that it is natural, and simpler, to work only in characteristic zero. We believe that in any case the statements are equivalent but have not bothered to check this carefully.}  in \cite{tao-edinburgh}. One of the main aims of this paper is to give a number of equivalent forms of the conjecture. Here is probably the simplest formulation. It is not the original one of Katz and Tao, which is Conjecture \ref{ak-conj} below. 

\begin{conjecture}\label{usp}
Let $k, N$ be positive integers. Write $F_k(N)$ for the size of the smallest set of integers containing, for each $d \in \{1,\dots, N\}$, a $k$-term arithmetic progression  with common difference $d$. Then 
\[ \lim_{k \rightarrow \infty} \lim_{N \rightarrow \infty} \frac{\log F_k(N)}{\log N} = 1.\]
\end{conjecture}

This conjecture was raised by the second author as \cite[Conjecture 4.2]{ruzsa-sumsets}, but no links to the Kakeya problem were mentioned there. \vspace{8pt}

We turn now to arguably the most natural of our formulations, concerning the entropy of random variables. As usual, the entropy $\bH$ of a random variable $\X$ with finite range is defined by
\[ \bH(\X) := -\sum_{x} \P(\X = x) \log \P(\X = x),\] where $x$ ranges over all values taken by $\X$.

\begin{conjecture}\label{eak}
Suppose that $\X$ and $\Y$ are two real-valued random variables, both taking only finitely many values. Then for any $\eps > 0$ there are\footnote{It is often convenient to ``work projectively'' and allow the $r_i$ to take values in $\Q \cup \{\infty\}$, where we define $\X + \infty \Y = \Y$. The two versions of Conjecture \ref{eak} are equivalent to one another, as may easily be seen by applying a projective transformation such as $\X' = (a + 1)\X$, $\Y' = a\X + \Y$ which preserves $\X - \Y$ but moves other rational combinations around.} $r_1,\dots, r_k \in \Q$, none equal to $-1$, so that 
\[ \bH(\X - \Y) \leq (1 + \eps) \sup_j \bH(\X + r_j \Y).\] 
\end{conjecture}

Next we give the original form of the conjecture discussed by Katz and Tao. Let $A \subset \Z \times \Z$ be a finite set. For rational $r$ we write $\pi_r(A) := \{ x + ry : (x,y) \in A\}$. We also write $\pi_{\infty}(A) := \{ y : (x,y) \in A\}$.

\begin{conjecture}\label{ak-conj} Let $\eps > 0$ be arbitrary. Then there are $r_1,\dots, r_k \in \Q \cup \{\infty\}$, none equal to $-1$, such that $\#\pi_{-1}(A) \leq \sup_{i} \# \pi_{r_i}(A)^{1 + \eps}$ for all finite sets $A \subset \Z \times \Z$.
\end{conjecture}

Our fourth conjecture has not, so far as we are aware, appeared explicitly in the literature before. It is in fact a whole family of conjectures, one for each natural number $n$; however, we will later show that all of these are equivalent.

\begin{conjecture}[$n$]\label{ff-conj}
Let $k$ be a positive integer. If $p$ is a prime, let $f_{k,n}(p)$ denote the size of the smallest set containing, for every $d \in \F_p^n \setminus \{0\}$, a $k$-term progression with common difference $d$. Then 
\[ \lim_{k \rightarrow \infty} \lim_{p \rightarrow \infty} \frac{\log f_{n,k}(p)}{\log p} = n.\]
\end{conjecture}
\emph{Remarks.} Note that $f_{p,n}(p)$ is the size of the smallest \emph{Besicovitch set} in $\F_p^n$, that is to say set containing a full line in every direction. Since $f_{p,n}(p) \geq f_{k,n}(p)$ whenever $p \geq k$, Conjecture \ref{ff-conj}($n$) trivially implies that 
\[ \lim_{p \rightarrow \infty} \frac{f_{p,n}(p)}{\log p} = n,\] i.e. any Besicovitch set in $\F_p^n$ has size $p^{n - o_{p\rightarrow \infty}(1)}$.
This is known to be true, a celebrated result of Dvir \cite{dvir}. However, the only known arguments use the ``polynomial method'' (see, for example, \cite{guth-polynomial, tao-polynomial} for modern introductions). This very strongly hints that any proof of Conjecture \ref{ff-conj} (and hence, by our main theorem, of the other conjectures) would have to use some form of the polynomial method.\vspace{8pt}

Our fifth and final conjecture is included mainly for historical interest, as it relates very closely to a question asked by Erd\H{o}s and Selfridge in the 1970s, well before the current wave of interest in the Kakeya problem and related matters.

\begin{conjecture}\label{es-conj}
Fix a positive integer $k$. Let $N$ be a positive integer. Then, uniformly for all $N$, all finite sets $p_1 < \dots < p_N$ of primes and all intervals $I \subset \N$ of length $kp_N$, we have
\[ \# \big( I \cap \bigcup_{i = 1}^N p_i \Z \big) \gg_k N^{1 - \gamma_k}.\] where $\gamma_k \rightarrow 0$ as $k \rightarrow \infty$.
\end{conjecture}
\emph{Remark.} Erd\H{o}s and Selfridge \cite[\S 6]{Erdos-1978} in fact asked whether or not one can take $\gamma_k = 0$. The second-named author \cite{ruzsa-few} showed that the answer is no, and in fact we must have $\gamma_k \geq \frac{1}{k}$. We note that Proposition \ref{prop4point1} and Theorem \ref{ck-lower} combine to give the much better bound $\gamma_k \gg \frac{1}{\log \log k}$.\vspace{8pt}

As previously stated, our main result is the equivalence of the five conjectures stated above.

\begin{theorem}\label{mainthm}
Conjectures \ref{usp}, \ref{eak}, \ref{ak-conj}, \ref{ff-conj}\textup{(}$n$\textup{)} \textup{(}for each $n = 1,2,3,\dots$\textup{)} and \ref{es-conj} are all equivalent.
\end{theorem}

Let us make some further remarks.

\begin{enumerate}
\item Once Theorem \ref{mainthm} is proven, it seems natural to use the term ``arithmetic Kakeya conjecture'' to refer to any one of the five conjectures.
\item It is known that Conjecture \ref{ak-conj} (and hence all the other conjectures) implies that the upper Minkowski dimension of any Besicovitch set\footnote{That is, a compact subset of $\R^n$ containing a unit line segment in every direction.} in $\R^n$ is $n$, a statement often referred to as the \emph{Kakeya conjecture}. This follows by a straightforward generalisation of the ``slicing'' argument of Bourgain \cite{bourgain-3}: a sketch of this may be found in \cite{tao-needle}. However, Bourgain \cite{bourgain-1,bourgain-2} observed that, in the notation of Conjecture \ref{usp}, the statement
\begin{equation}\label{ulp} \lim_{N \rightarrow \infty} \frac{\log F_{N^{\eta}}(N)}{\log N} \geq 1\end{equation} for all $\eta > 0$ also implies the Kakeya conjecture. Since $F_k(N)$ is a nondecreasing function of $k$, \eqref{ulp} is immediately implied by Conjecture \ref{usp}, whilst an implication in the reverse direction seems very unlikely without resolving both conjectures. In this sense, the arithmetic Kakeya conjecture should be considered a \emph{strictly} harder problem than the Kakeya conjecture.
\item The equivalence of Conjectures \ref{eak} and \ref{ak-conj} was proven by the second author in \cite{ruzsa-sumsets-entropy} (see also \cite{lemm}). We are not aware of any references for the other implications.
\end{enumerate}

Now we discuss the other results in the paper. First, we establish a lower bound showing that the convergence in Theorem \ref{usp}, if it occurs, is very slow.

\begin{theorem}\label{ck-lower}
In the notation of Conjecture \ref{usp}, we have
\[ \lim_{N \rightarrow \infty} \frac{\log F_k(N)}{\log N} \leq 1- \frac{c}{\log \log k},\] where the constant $c > 0$ is absolute.
\end{theorem}

Second, we show that a finite field variant of Conjecture \ref{eak} \emph{is} true. Write $\F_p^{\infty}$ for the vector space over $\F_p$ of countably infinite dimension.

\begin{theorem}\label{thm13}
Suppose that $\X$ and $\Y$ are two $\F_p^{\infty}$-valued random variables, both taking only finitely many values. Then
\[ \bH(\X - \Y) \leq (1 + O(\frac{1}{\log p})) \sup_{r \in \F_p \cup \{\infty\} \setminus \{-1\}} \bH(\X + r \Y).\] Here, the constant in the $O()$ notation is absolute.
\end{theorem}

The $O(\frac{1}{\log p})$ term is best possible, as we remark in \S \ref{fin-field-ent}. \vspace{8pt}

We neither discuss nor make progress on partial results towards any of Conjectures \ref{usp}, \ref{eak}, \ref{ak-conj}, \ref{ff-conj} or \ref{es-conj}. We believe that the best value of $\eps$ for which Conjecture \ref{eak} is known is $\eps \approx 0.67513\dots$, which is equivalent to a result obtained in \cite{katz-tao-2}. (The precise value here is $\alpha - 1$, where $\alpha$ solves $\alpha^3 - 4\alpha + 2 = 0$.) This bound is now 15 years old.\vspace{8pt}

\emph{Notation.} Most of our notation is quite standard. We use $\# X$ for the cardinality of a set $X$. Occasionally, if $A$ is a set in some abelian group and $k$ is an integer we will write $k \cdot A$ to mean $\{ ka : a \in A\}$.\vspace{8pt}

\emph{Acknowledgements.} The first author is supported by a Simons Investigator Grant, and is very grateful to the Simons Foundation for this support.

\section{Progressions, projections and entropy}

In this section we establish around half of Theorem \ref{mainthm} by proving that the first three conjectures mentioned in the introduction are equivalent. Whilst at a local level the arguments are a mix of fairly unexciting linear algebra and standard tools such as Freiman isomorphisms, random projections and taking tensor powers, the large number of them makes the proof of Theorem \ref{mainthm} somewhat lengthy.

It is convenient to proceed by first showing that Conjectures \ref{usp}, \ref{ak-conj} and \ref{eak} are equivalent. In the course of doing so, and for later use, it is convenient to introduce a further conjecture, apparently stronger than Conjecture \ref{usp} but, as it turns out, equivalent to it.

\begin{usp-dash}
Let $k$ be a positive integer. Write $F'_k(N)$ for the cardinality of the smallest set $A \subset \Z$ which contains an arithmetic progression of length $k$ and common difference $d$, for $N$ different values of $d$. Then
\[ \lim_{k \rightarrow \infty}\lim_{N \rightarrow \infty} \frac{\log F'_k(N)}{\log N} = 1.\]
\end{usp-dash}

It is obvious that Conjecture \ref{usp}' implies Conjecture \ref{usp}, because $F'_k(N) \leq F_k(N)$. It turns out that the reverse implication holds as well. In fact, we claim that the following is true.
\begin{proposition}\label{prop21} We have $F_k(N) \ll k^3 \log N \cdot F'_k(N)$.\end{proposition}
\begin{proof}
Suppose we have a set
\[ A_0 = \bigcup_{i = 1}^N \bigcup_{j = 0}^{k-1}\{ a_i + jd_i\},\] where the $d_i$ are distinct. We claim that there is a set $A_1$, $\# A_1 \ll k^3 \log N \cdot \# A_0$, containing an arithmetic progression of length $k$ and common difference $d$ for all $d \in \{1,\dots, N\}$. This obviously implies the result.

Pick $\theta \in (0,1)$ uniformly at random, and define the function 
\[ \phi_{\theta} : \Z \rightarrow \{0,1,\dots,N-1\}\] by
\[ \phi_{\theta}(x) := \lfloor N \{ \theta x\} \rfloor.\]
Here, $\{t\} = t - \lfloor t \rfloor$, so $0 \leq \{t \} < 1$.

Note that if $i \neq j$ then
\[ \P_{\theta}(\phi_{\theta}(d_i) = \phi_{\theta}(d_j)) \leq \P_{\theta}(\theta(d_i - d_j) \in (-\frac{1}{N}, \frac{1}{N}) \md{1}) = \frac{2}{N}.\]
It follows that the expected number of pairs $(i,j)$ with $i < j$ for which $\phi_{\theta}(d_i) = \phi_{\theta}(d_j)$ is at most $\frac{2}{N} \binom{N}{2} = N - 1$. By linearity of expectation, there is some choice of $\theta$ for which, setting $d'_i := \phi_{\theta}(d_i)$, there are at most $N - 1$ pairs $(i,j)$ with $i < j$ and $d'_i = d'_j$. If $n \in \{0,1,\dots, N-1\}$, write $f(n)$ for the number of $i$ with $d'_i = n$. Then it follows that $\sum_n \binom{f(n)}{2} \leq N - 1$, from which we obtain, since $\sum_n f(n) = N$, that $\sum_n f(n)^2 \leq 3N$. By Cauchy-Schwarz, 
\[ N^2 = (\sum_n f(n))^2 \leq \# \{n : f(n) \neq 0\} \sum_n f(n)^2,\]
and therefore there are at least $N/3$ values of $n$ for which $f(n) \neq 0$, or in other words there are at least $N/3$ distinct values amongst the $d'_i$.

Now consider the set $A_2 := \phi_{\theta}(A_0)$. Obviously $\# A_2 \leq \# A_0$. Whilst $A_2$ itself does not obviously contain any long progressions, we observe that 
\[ \phi_{\theta}(a_i + (j+1) d) - \phi_{\theta}(a_i + j d) - d'_i  \in \{0,1\} - \{0,N\}\] (In fact, $\phi_{\theta}(x + y) - \phi_{\theta}(x) - \phi_{\theta}(y) \in \{0,1\} - \{ 0,N\}$ for every $x,y$.) By a simple induction, 
\[ \phi_{\theta}(a_i) + jd'_i - \phi_{\theta}(a_i + jd) \in \{0,1,\dots, k-1\} - \{ 0, N, \dots, (k-1) N\}\] for $j = 0,1,\dots, k-1$, and so the set $A_3 := A_2 + \{0,1,\dots, k-1\} - \{ 0, N, \dots, (k-1) N\}$ contains a progression of length $k$ and common difference $d'_i$, for all $i$. Note that $\# A_3 \leq k^2 \# A_0$. 

By taking random translates (see Lemma \ref{add-comp} for details) and the fact that there are $\geq N/3$ distinct $d'_i$, there is some set $T$ of integers, $\# T \ll \log N$, such that every element of $\{1,\dots, N\}$ can be written as $d'_i + t$ with $t \in T$. Set
\[ A_1 := A_3 + \{0,1,\dots, k-1\} \cdot T.\]
We have $\# A_1 \leq k \cdot \# T \cdot \# A_3 \ll k^3 \log N \cdot \# A_0$. It is easy to see that $A_1$ contains an arithmetic progression of length $k$ and common difference $d'_i + t$, for all $i$ and for all $t \in T$, and hence contains an arithmetic progression of length $k$ and common difference $d$ for all $d \in \{1,\dots, N\}$. This concludes the proof of Proposition \ref{prop21}.
\end{proof}

Now we turn to the proof that Conjectures \ref{usp}', \ref{eak} and \ref{ak-conj} are equivalent. 

\emph{Conjecture \ref{usp}' implies Conjecture \ref{ak-conj}.} Suppose that Conjecture \ref{ak-conj} is false. Then there is some $\eps > 0$ such that, for every $k$, there is a set $A_k \subset \Z \times \Z$ such that 
\[ \# \pi_{-1}(A_k) > \max_{r \in H_k \setminus \{-1\}} \# \pi_r(A_k)^{1 + \eps},\] where $H_k$ denotes the set of rationals with height at most $k$, that is to say \[ H_k := \{ \frac{a}{b} : |a|, |b| \leq k\} \cup \{\infty\}.\] 
Our first step is to use a ``tensor power'' argument to show that there are arbitrarily large sets with the same property; in fact, we shall argue that for every $j$ there is a set $A_{k,j} \subset \Z \times \Z$ such that 
\begin{equation}\label{blowup} \# \pi_{-1}(A_{k,j}) \geq  j\max_{r \in H_k \setminus \{-1\}} \# \pi_r(A_{k,j})^{1 + \eps}.\end{equation}
This is simple if the $A_{k,j}$ are allowed to be subsets of $\Z^n$. Indeed we may define $A_k^{(n)}$ to be the set
\[  \{ \big( (a_1, a_2,\dots, a_n), (a'_1,a'_2,\dots, a'_n) \big) \in \Z^n \times \Z^n : (a_i, a'_i) \in A_k \; \mbox{for all $i$}\}.\] Then, writing $\pi_r^{(n)} : \Z^n \times \Z^n \rightarrow \Z^n$, for the map sending $(x,y)$ to $x + ry$ (or, when $r = \infty$, to $y$) we have 
\[ \# \pi_r^{(n)}(A_k^{(n)}) = \big(  \# \pi_r(A_k) \big) ^n \] for all $r,n$. In particular, by choosing $n$ large enough (depending on $j$) we have
\begin{equation}\label{tens} \# \pi_{-1}^{(n)}(A_k^{(n)}) \geq j  \max_{r \in H_k \setminus \{-1\}} \# \pi_r^{(n)} (A_k^{(n)})^{1 + \eps}.\end{equation}

To create a subset of $\Z \times \Z$ from $A_k^{(n)}$, we apply a map $\psi_t : \Z^n \times \Z^n \rightarrow \Z \times \Z$ of the form \[ \psi_t(x,y) = ((t, t^2,\dots, t^n) \cdot x, (t, t^2,\dots, t^n) \cdot y)),\] where the dot denotes the usual inner product. Setting $A := \psi_t (A_k^{(n)})$, we have 
\[ \pi_r(A) = \psi_t(\pi_r^{(n)}(A_k^{(n)})).\]
Choose $t$ to be an integer such that for $r \in H_k$ and $(x,y), (x',y') \in A_k^{(n)}$ we have
\begin{equation}\label{xyxy} (\pi_r^{(n)}(x,y) - \pi_r^{(n)}(x',y')) \cdot (t, t^2,\dots, t^n) \neq 0\end{equation} unless $\pi_r^{(n)}(x,y) = \pi_r^{(n)}(x',y')$. There is such a $t$, because for each of the finite number of choices of $x,y,x',y',r$ the left-hand side of \eqref{xyxy} is a nontrivial polynomial equation in $t$. It then follows that $\pi_r (\psi_t(x,y)) = \pi_r(\psi_t(x',y'))$ if and only if $\pi_r^{(n)}(x,y) = \pi_r^{(n)}(x', y')$, and so 
\[ \# \pi_r(A) = \# \pi_r^{(n)}(A_k^{(n)})\] for all $r$. 
This establishes the existence of the sets $A_{k,j}$ satisfying \eqref{blowup}.

For each $j,k$, consider the set $S_{k,j} \subset \Q$ defined by
\[ S_{k,j} := \bigcup_{1 \leq i \leq k} \bigcup_{r \in H_k \setminus \{-1\}} \frac{i}{k} \cdot \pi_r(A_{k,j}).\] Then
\[ \# S_{k,j} \leq k \cdot \# H_k \cdot \max_{r \in H_k^+} \# \pi_r(A_{k,j}) \ll_k \# (j^{-1}\pi_{-1}(A_{k,j}))^{1/(1 + \eps)}.\]
On the other hand, suppose that $d \in -\pi_{-1}(A_{k,j})$. This means that $d = y - x$ for some $(x,y) \in A_{k,j}$. If $0 \leq i \leq k - 1$, we have
\[ x + \frac{id}{k} =  \frac{k-i}{k}\big(x + \frac{i}{k-i} y\big) .\]
Since $x + \frac{i}{k-i} y \in \pi_{i/(k-i)}(A_{k,j}) \subset \bigcup_{r \in H_k^+} \pi_r(A_{k,j})$, it follows that $x + \frac{id}{k} \in S_{k,j}$ for  $i = 0,1,\dots, k-1$, that is to say $S_{k,j}$ contains a progression of length $k$ and common difference $\frac{d}{k}$. Thus, writing $N_j := \# \pi_{-1}(A_{k,j})$, we see that $S_{k,j}$ is a set of size $\ll (j^{-1} N_j)^{1/(1 + \eps)}$ containing progressions of length $k$ with at least $N_j$ distinct common differences. Since, evidently, $\# S_{k,j} \geq k$, the presence of the factor $j^{-1}$ forces $N_j \rightarrow \infty$ as $j \rightarrow \infty$. By multiplying through by an appropriate integer, we may find sets $\tilde S_{k,j} \subset \Z$ with the same property, contrary to Conjecture \ref{usp}'.\vspace{8pt}

\emph{Conjecture \ref{ak-conj} implies Conjecture \ref{eak}.} This implication is essentially given in \cite{ruzsa-sumsets-entropy}. The notation there takes a little unpicking and the proof is short, so we repeat the argument.

Let $\eps > 0$, and suppose that $r_1,\dots, r_k \in \Q_{\geq 0} \cup \{\infty\} \setminus \{-1\}$ are such that \begin{equation}\label{conj2-assump}\# \pi_{-1}(A) \leq \sup_i \# \pi_{r_i}(A)^{1 + \eps}\end{equation} for all finite sets $A \subset \Z \times \Z$. We claim that 
\begin{equation}\label{ent-claim} \bH(\X - \Y) \leq (1 + \eps)\sup_j \bH(\X + r_j \Y).\end{equation}
for all $\Z$-valued random variables $\X, \Y$, both taking only finitely many values. (Let us remind the reader that, by convention, $\bH(\X + \infty \Y) = \bH(\Y)$.)

We begin with a couple of observations. The first is that \eqref{conj2-assump} is automatically true for sets $A \subset \Z^n \times \Z^n$, for any $n$. This follows from the case $n = 1$ by applying a suitable map $\psi_t : \Z^n \times \Z^n \rightarrow \Z \times \Z$, exactly as in the argument following \eqref{tens} above.

The second observation is that, by a simple limiting argument, we may assume that there is some $q$ such that $q\P((\X, \Y) = (x,y)) \in \Z$ for all $(x,y)$: if we can prove the result for such $(\X, \Y)$ the same inequality for arbitrary $\X, \Y)$ with finite range follows by letting $q \rightarrow \infty$.

Now let $m$ be very large, and construct a set $A \subset \Z^{mq} \times \Z^{mq}$ as follows. Let it consist of all pairs $((x_1,\dots, x_{mq}), (y_1,\dots, y_{mq})) \in \Z^{mq} \times \Z^{mq}$ for which 
\[ \# \{ i : (x_i, y_i) = (x,y)\} = mq\P((\X, \Y) = (x,y)).\]
Let us calculate $\# \pi_{r}(A)$. After a moment's thought we see that 
\[ \pi_r(A) = \big\{ (z_1,\dots, z_{mq}) : \# \{i : z_i = z\} = mq \P(\X + r\Y = z)\big\}.\] (Here, we interpret $\P(X + \infty \Y = z)$ as $\P(\Y = z)$.) Writing $n = mq$ and $p_z = \P(\X + r\Y = z)$ for short, it follows that 
\[ \# \pi_r (A) = \frac{n!}{\prod_{z} (n p_z)!}.\]
Note that the product over $z$ is finite, and that each $n p_z$ is an integer. Taking logs and using the fact that $\log N! = N \log N - N + o(N)$, we have
\[ \log \pi_r(A) = -n \sum_z p_z \log p_z + o(n) = n \bH(\X + r\Y) + o(n).\]
We may assume that the $o(n)$ term is uniform in $r \in \{r_1,\dots, r_k\}$ (since this is a finite set); of course, it also depends on $\X, \Y$, but we are thinking of these as fixed for the duration of the argument.

Taking logs of \eqref{conj2-assump} (which is valid for $A \subset \Z^{n} \times \Z^n$, as remarked), we conclude that 
\[ n\bH(\X - \Y) \leq (1 + \eps)n  \sup_i \bH(\X + r_i\Y) + o(n).\]
Now we may simply divide through by $n$ and let $n \rightarrow \infty$ to conclude the claim \eqref{ent-claim}.
\vspace{8pt}

\emph{Conjecture \ref{eak} implies Conjecture \ref{usp}.} This is relatively easy. Assume Conjecture \ref{eak}. Let $\eps > 0$ be arbitrary, and select $r_1,\dots, r_m \in \Q \cup \{\infty\} \setminus \{-1\}$ so that we have
\begin{equation}\label{conj2-rpt} \bH(\X - \Y) \leq (1 + \eps) \sup_i \bH(\X + r_i \Y).\end{equation}
Let $Q, M$ be positive integers to be specified later (depending on $r_1,\dots$, $r_m$) and suppose that $A \subset \Z$ contains an arithmetic progression of length $k = 2MQ$ and common difference $d$, for every $d \in \{1,\dots, N\}$. Define $\Z$-valued random variables $\X$, $\Y$ as follows: pick $d$ uniformly at random, and let $\{a(d),\dots, a(d) + (k-1) d\}$ be the progression in $A$ for which $a(d)$ is minimal (choosing $a(d)$ minimal is not important, but is one way of making a definite choice). Set $\X = a(d) + MQd$ and $\Y = a(d) + (M+1)Qd$.

Then $\X - \Y$ is uniformly distributed on the set $\{-Q, -2Q, \dots, -NQ\}$, and so 
\begin{equation}\label{unixy} \bH(\X - \Y) = \log N.\end{equation} 

On the other hand,
\[ \bH(X + r_j \Y) = \bH\big( \frac{\X + r_j \Y}{1 + r_j} \big) = \bH \big(a(d) + (QM + \frac{Qr_j}{1 + r_j}) d \big).\]
By choosing $Q$ and then $M$ suitably, we may ensure that all the $Qr_j/(1 + r_j)$ are integers of magnitude $< QM$, which means that 
\[ a(d) + (QM + \frac{Qr_j}{1 + r_j}) d \in \{a(d),\dots, a(d) + (k-1) d\} \subset A.\]
That is, $\X + r_j \Y$ takes values in $(1 + r_j) \cdot A$. Since $\bH(\mathsf{W}) \leq \log m$ for any random variable $\mathsf{W}$ taking values in a set of size $m$, this implies that 
\[ \bH(\X + r_j \Y) \leq \log \# A.\] Combining this with \eqref{conj2-rpt} and \eqref{unixy} we obtain
\[ \log N \leq (1 + \eps) \log \# A,\] or in other words
\[ \# A \geq N^{1/(1 + \eps)}.\]
Since $\eps$ was arbitrary, the implication follows.\vspace{11pt}

This completes the proof that Conjectures \ref{usp}, \ref{usp}', \ref{eak} and \ref{ak-conj} are equivalent

\section{Finite fields} Next we turn to Conjecture \ref{ff-conj} ($n$). To demonstrate its equivalence to the first three conjectures, it suffices to show that for each $n$ we have Conjecture \ref{usp}' $\Rightarrow$ Conjecture \ref{ff-conj}($n$) $\Rightarrow$ Conjecture \ref{usp}.

\emph{Conjecture \ref{usp}' implies Conjecture \ref{ff-conj} \textup{(}$n$\textup{)}.} Suppose that $A_1 \subset \F_p^n$ is a set containing a $k$-term arithmetic progression with common difference $d$, for every $d \in \F_p^n$.  Define the ``unwrapping'' map $\psi : \F_p \rightarrow \Z$ to be the inverse of the natural projection map from $\{0,\dots,p-1\}$ to $\F_p$. Define a map $\psi^{(n)} : \F_p^n \rightarrow \Z^n$ by setting $\psi^{(n)}(x_1,\dots, x_n) := (\psi(x_1),\dots, \psi(x_n))$.

For each $d \in \F_p^n$, select a progression $\{ x(d) + \lambda d, \lambda = 0,1,\dots, k-1\}$, lying in $A_1$. Let $A_2 \subset \Z^n$ be the union of all progressions $\{ \psi^{(n)}(x(d)) + \lambda \psi^{(n)}(d) : \lambda = 0,1,\dots, k-1\}$. By construction, $A_2 \subset \{0,1,\dots, k(p-1)\}^n$, and $\pi^{(n)}(A_2) \subset A_1$, where $\pi^{(n)} : \Z^n \rightarrow \F_p^n$ is the natural map. Since $\{ 0,1,\dots, k(p-1)\}$ is covered by $k$ discrete intervals of length $p$, on each of which the projection map $\pi : \Z \rightarrow \F_p$ is injective, we see that $\# A_2 \leq k^n \# A_1$.

By construction, $A_2$ contains a progression of length $k$
and common difference $d$ for $p^n$ distinct values of $d$. Whilst $A_2$ is a subset of $\Z^n$, we can create a subset of $\Z$ with the same properties by looking at the image of $A_2$ under the map $f : \Z^n \rightarrow \Z$ defined by $f(x_1,\dots, x_n) = \sum_{i = 1}^n (10 kp)^i x_i$. It follows that $\# A_2 \geq F'_k(p^n)$, and hence $\# A_1 \geq k^{-n} F'_k(p^n)$. In the notation of Conjecture \ref{ff-conj}, this means that $f_{n,k}(p) \geq k^{-n} F'_k(p^n)$. It follows that 
\[ \lim_{p \rightarrow \infty} \frac{\log f_{n,k}(p)}{\log p} \geq  n \lim_{p \rightarrow \infty} \frac{\log F'_k(p^n)}{\log p^n},\] and so
\[ \lim_{k \rightarrow \infty}\lim_{p \rightarrow \infty} \frac{\log f_{n,k}(p)}{\log p} \geq  n \lim_{k \rightarrow \infty}\lim_{p \rightarrow \infty} \frac{\log F'_k(p^n)}{\log p^n}.\]
Assuming Conjecture \ref{usp}' (taking $N = p^n$), the right hand side here is precisely $n$. This implies Conjecture \ref{ff-conj}.\vspace{8pt}

\emph{Conjecture \ref{ff-conj} implies Conjecture \ref{usp}.} Suppose we have a set $A_1 \subset \Z$ containing a progression of length $k$ and common difference $d$ for each $d \in \{1,\dots, N\}$. Partition $\Z$ into intervals $I_j := 10 kj N + \{1,\dots, 10k N\}$, $j \in \Z$. Any progression of length $k$ and common difference $d \in \{1,\dots,N\}$ is either wholly contained in some $I_j$, or else is split into two progressions, one in $I_j$ and the other in $I_{j+1}$, with one of these having length at least $k/2$. It follows that the set $A_2 \subset \Z$ defined by\footnote{This ``cut-and-move'' trick is quite standard in the study of the Kakeya problem.}
\[ A_2 = \bigcup_j \{ (A_1 \cap I_j) - 10 kj N\}\] contains a progression of length at least $k/2$ and common difference $d$, for all $d \in \{1,\dots, N\}$. Manifestly $\# A_2 \leq \# A_1$, and by construction $A_2$ has the additional property that
\begin{equation}\label{compact} A_2 \subset \{1,\dots, 10k N\}.\end{equation}

Using $A_2$, we construct a set $A_3 \subset \Z^n$. We will later use this to construct a further set $A_4 \subset \F_p^n$, for a suitable prime $p$, by projection. To define $A_3$, let $M := \lfloor N^{1/n}\rfloor$. Select $t \in \{-10kN,\dots, 20kN - 1\}$ uniformly at random, and define
\[ A_3(t) := \{(x_1,\dots, x_n) \in \{0,\dots, M-1\}^n : \sum_{i=1}^n M^{i - 1} x_i \in A_2 + t\}.\]
Suppose that $d = \sum_{i = 1}^n M^{i - 1} d_i$ with $0 \leq d_i \leq M/2k$ for all $i$. There are at least $(M/4k)^n$ such values of $d$, and all lie in $\{0,\dots, N\}$. For each such $d$ there is, by assumption, a progression $\{ x(d) + \lambda d : \lambda = 0,1,\dots, \lfloor k/2\rfloor - 1\}$ lying in $A_2$. The progression $\{ x(d) + t + \lambda d : \lambda = 0,1,\dots, \lfloor k/2\rfloor - 1\}$ then lies in $A_2 + t$. Write
\[ S := \{ \sum_{i = 1}^n M^{i - 1} s_i : 0 \leq s_i < M/2 \; \mbox{for all $i$}\}.\]

If it so happens that $t \in - x(d) + S$ then $A_3(t)$ contains a progression of length $k$ and common difference $(d_1,\dots, d_n)$, namely $\{ (s_1,\dots, s_n) + \lambda(d_1,\dots, d_n) : \lambda \in\{0,1,\dots, k - 1\}\}$, where $x(d) + t = \sum_{i = 1}^n M^{i - 1} s_i$.

Since $0 \leq x(d) \leq 10kN$ and $S \subset \{0,1,\dots, M^n\}$, $-x(d) + S \subset \{-10k N,\dots, 20k N - 1\}$. It follows that 
\[ \P(t \in -x(d) + S) = \frac{1}{30kN} \# S \geq \frac{1}{30kN}(\frac{M}{2})^n \gg_{k,n} 1.\]
Summing over the $(M/2k)^n \gg_{k,n} N$ choices of $d$, we see that the expected number of $d$ for which $t \in -x(d) + S$ is $\gg_{k,n} N$. Fix some choice of $t$ such that $t \in -x(d) + S$ for $\gg_{k,n} N \gg_{k,n} M^n$ values of $d$, and write $A_3 := A_3(t)$. Then by construction we have
\begin{equation}\label{triv-compar} \# A_3 \leq \# A_2 \leq \# A_1,\end{equation}
whilst $A_3$ contains a progression of length $\geq k/2$ and common difference $d$ for all $d$ in some set $\mathscr{D} \subset \{0,\dots, M-1\}^n$, $\# \mathscr{D} \gg_{k,n} M^n$.

Now choose a prime $p$ with $M \leq p < 2M$, and let $A_4 \subset \F_p^n$ be the image of $A_3$ under the natural projection $
\pi^{(n)} : \Z^n \rightarrow \F_p^n$. We have 
\begin{equation}\label{triv-compare-2} \# A_4 = \# A_3,\end{equation}
and moreover $A_4$ contains a progression of length $k$ and common difference $d$ for all $d \in \pi^{(n)}(\mathscr{D})$, that is to say for $\gg_{n,k} N \gg_{n,k} p^n$ values of $d$. By a standard argument (taking random translations of $\pi^{(n)}(\mathscr{D})$, see Corollary \ref{a3} for details) there is a further set $A_5 \subset \F_p^n$,  
\begin{equation}\label{triv-compare-3} \# A_5 \ll_{n,k} (\log p) \# A_4,\end{equation} containing a progression of length $k$ and common difference $d$, for \emph{all} $d \in \F_p^n \setminus \{0\}$. 
Tracing back through \eqref{triv-compare-3}, \eqref{triv-compare-2}, \eqref{triv-compar} we see that 
\[ F_k(N) \gg_{k,n} \frac{1}{\log p} f_{n,k}(p),\]
where $p = p(N) \sim N^{1/n}$ is some prime. It follows that 
\[ \lim_{N \rightarrow \infty}\frac{\log F_k(N)}{\log N} \geq \lim_{N \rightarrow \infty}\frac{\log f_{n,k}(p(N))}{n \log p(N)}.\]
Assuming Conjecture \ref{ff-conj} ($n$), the limit on the right is $1$. This concludes the proof that Conjecture \ref{ff-conj} ($n$) implies Conjecture \ref{usp}.\vspace{8pt}

Before leaving this topic, we remark that it is quite possible that in the regime $\log k \asymp \log p$ very strong bounds such as 
\begin{equation}\label{pp-conj} f_{p^{\eta}, 1}(p) \geq p/2\end{equation} are true, provided $p \geq p_0(\eta)$ is large enough. This issue is strongly hinted at, if not explicitly conjectured, in \cite{alon-peres}. It is pointed out there that such bounds imply vastly more than is currently known about the purely arithmetic problem of bounding the least quadratic nonresidue modulo $p$.

Whilst a bound of this type is not known to imply the arithmetic Kakeya conjecture (the progressions are of length $p^{\eta}$, rather than of bounded size), the arguments of Bourgain may be adapted to show that it does  imply the Kakeya conjecture. Further details may be found in lecture notes of the first author \cite[Section 10]{green-rkp}.

It is quite interesting that the innocent-looking statement \eqref{pp-conj} implies two famous unsolved problems in completely different mathematical areas.

\section{A problem of Erd\H{o}s and Selfridge} Finally, we turn to Conjecture \ref{es-conj}.  In fact, we prove the following rather tight connection between Conjecture \ref{usp}' and Conjecture \ref{es-conj}.

\begin{proposition}\label{prop4point1}
Write $G_k(N)$ for the minimum, over all intervals $I$ of length $k p_N$ and all choices $p_1 < \dots < p_N$ of primes, of $\# \big( I \cap \bigcup_{i = 1}^N p_i \Z\big)$. Then $F'_k(N) \leq G_k(N) \leq k F'_k(N)$. In particular, Conjectures \ref{usp}' and \ref{es-conj} are equivalent.
\end{proposition}
\begin{proof}
Suppose first we have a set of primes $p_1 < \dots < p_N$ and an interval $I$ of length $kp_N$ so that $\# A = G_k(N)$, where $A = \bigcup_{i = 1}^N\{ x \in I : p_i | x \}$. Note that $A$ obviously contains a progression of length $k$ and common difference $p_i$, for each $i$, and therefore $F'_k(N) \leq G_k(N)$.

In other other direction, suppose we have a set $A$ attaining the bound $F'_k(N)$, that is to say $\# A = F'_k(N)$ and $A$ contains, for $i = 1,\dots, N$, a progression $\{ a_i + jd_i : j = 0,1,\dots, k-1\}$. By translating if necessary, we may assume that $A$ consists of positive integers. Let $\delta \in (0,\frac{1}{2})$ be a quantity to be specified shortly. By the theorem of the first author and T.~Tao \cite[Theorem 1.2]{green-tao}, we may find positive $u$ and $v$ such that all of the numbers $v, u+v, \dots d_N u + v$ are prime and lie in some interval $[(1 - \delta) X, X]$, $X \geq 100$. Set $p_i := d_i u + v$. Note that $\frac{v}{u+v} \geq 1 - \delta$, which rearranges as $\frac{v}{u} \geq \frac{1}{\delta} - 1$, hence
\begin{equation}\label{eq32} \frac{v}{u} >  4\max A\end{equation}  provided that $\delta$ is chosen sufficiently small.
Note also that 
\begin{equation}\label{eq35} \frac{p_i}{p_N}\geq \frac{v}{v + ud_N} = \frac{1}{1 + \frac{u}{v} d_N} \geq 1 - \frac{1}{4k}\end{equation} if $\delta$ is small enough. In particular if $\delta$ is small enough then we have 
\begin{equation} \label{eq34}p_i >  \frac{3}{4}p_N \geq \frac{1}{2} p_N + \frac{1}{4}v > \frac{1}{2}p_N + u \max A\end{equation} by \eqref{eq32}.

Define $A' := u \cdot A + \{0,v,2v,\dots, (k-1)v\}$. The cardinality of $A'$ satisfies $\# A' \leq k F'_k(N)$, and 
\begin{equation}\label{a-cont} A' \supset \bigcup_{i = 1}^N\{ ua_i + jp_i : j = 0,1,\dots, k-1\}\end{equation} for $i = 1,\dots, N$. By the Chinese remainder theorem we may find $w$ so that $p_i | w + ua_i$ for $i = 1,\dots, N$. 

Set $I := w - \lfloor \frac{1}{2} p_N\rfloor + \{1,2,\dots, kp_N\}$. Obviously $I$ is an interval of length $kp_N$. Let $i \in \{1,\dots, N\}$. We claim that $w + ua_i + jp_i \in I$ for an integer $j$ if and only if $j \in \{0,1,\dots, k-1\}$. Since $w + ua_i + p_i\Z  = p_i \Z$, this implies that
\[ I \cap p_i \Z = \{ w + ua_i + jp_i : j = 0,1,\dots, k-1\},\] and hence by \eqref{a-cont}
\[ I \cap \bigcup_{i = 1}^N p_i \Z \subset w + A',\] whence
\[ G_k(N) \leq \# \big( I \cap \bigcup_{i = 1}^N p_i \Z   \big) \leq \# A' \leq k F'_k(N).\]
It remains to prove the claim. To prove the \emph{if} implication, it suffices in view of \eqref{a-cont} to show that $w + A' \subset I$. However it is obvious that $\min (w + A') \geq \min I$ (since all elements of $A'$ are positive) and moreover
\begin{align*}
\max(w + A') & \leq w + u \max A + (k-1) v \\ & < w + (k- \frac{1}{2}) v \;\; \mbox{by \eqref{eq32}} \\ & \leq w + (k - \frac{1}{2}) p_N \\ & \leq \max I.
\end{align*} 
This establishes the \emph{if} direction of the claim. To establish the \emph{only if} direction, it suffices to show that 
\begin{equation}\label{claim1}  w + ua_i - p_i < \min I\end{equation} and that 
\begin{equation}\label{claim2} w + ua_i + kp_i > \max I.\end{equation}
However by \eqref{eq34} we have
\[ w + ua_i - p_i < w + u(a_i - \max A) - \frac{1}{2}p_N \leq w - \frac{1}{2} p_N \leq \min I,\] so \eqref{claim1} does hold. Also,
\begin{align*}
w + ua_i + kp_i & > w + kp_i \; \; \mbox{since $A \subset \N$} \\ & \geq w + (k - \frac{1}{4}) p_N \; \; \mbox{by \eqref{eq35}} \\ & \geq w + kp_N - \lfloor \frac{1}{2} p_N \rfloor = \max I,
\end{align*}
the last step being a consequence of the fact that $p_N \geq (1-\delta) X \geq 50$.
Thus \eqref{claim2} also holds, and this completes the proof of the claim.
\end{proof}

\emph{Remark.} The use of the theorem of the first author and Tao is a little excessive. One could do without it using simpler arguments if one was prepared to settle for logarithmic losses.

\section{Small unions of progressions}

In this section we prove Theorem \ref{ck-lower}. Write $3 = p_1 < p_2 < \dots$ for the odd primes, and set $Q := \prod_{i = 1}^m p_i$, where $m = \lceil 10 \log k\rceil$. Note that $Q = k^{O(\log \log k)}$. 

Define a set $S$ to be the union of all progressions $\{ x_d + j d : j = 0,1,\dots, k-1\}$ where, for $d \in \{1,\dots, Q-1\}$, $x_d$ is the unique element of $\{1,\dots,Q\}$ congruent to $d^2 \md{Q}$. Evidently, $S$ contains a progression of length $k$ and common difference $d$, for all $d \in \{0,1,\dots, Q-1\}$.

Fix $j \in \{0,1\dots, k-1\}$. For each $i$ we have
\[ x_d + jd \equiv d^2 + jd \equiv (d + \frac{j}{2})^2 - \frac{j^2}{4} \md{p_i},\] and so $x_d + jd \md{p_i}$ takes values in a set of size $\frac{1}{2}(p_i + 1)$ as $d$ varies. Therefore
$x_d + jd \md{Q}$ takes values in a set of size $\prod_{i = 1}^m \frac{1}{2}(p_i + 1)$. Since, additionally, $0 < x_d + jd \leq kQ$, $x_d + jd$ takes values in a set of size $k \prod_{i = 1}^m\frac{1}{2}(p_i + 1)$. Therefore
\[ \#S \leq k^2 \prod_{i = 1}^m \frac{1}{2}(p_i + 1) = k^2 2^{-m} Q \prod_{i = 1}^m (1 + \frac{1}{p_i}). \]
Recalling that $m \sim 10\log k$, and using the bound $\prod_{i = 1}^m (1 + \frac{1}{p_i}) \ll \log m \lll k$, we see that 
\[ \# S \ll k^{-7} Q \] and so \[ \# S \leq Q^{1 - \frac{c}{\log \log k}}\] if $k$ is sufficiently large, for some absolute $c > 0$.

Now let $n$ be an arbitrary positive integer, set $N_n := Q^n$, and consider the set
\[ A_n := \{ s_0 + s_1 Q + \dots + s_{n-1} Q^{n-1} : s_0,\dots, s_{n-1} \in S\}.\] Then $\# A_n \leq (\# S)^n \leq N_n^{1 - \frac{c}{\log \log k}}$.
The set $A_n$ contains a progression of length $k$ and common difference $d_0 + d_1 Q + \dots + d_{n-1}Q^{n-1}$ for any choice of $d_i \in \{0,1,\dots, Q-1\}$, or in other words for all $d \in \{0,\dots, N_n - 1\}$.

Finally, suppose $N$ is an arbitrary positive integer. Choose $n$ minimal so that $N_n > N$, and set $A := A_n$. Then $A$ contains a progression of length $k$ and common difference $d$, for all $d \in \{1,\dots, N\}$. Moreover, 
\[ \# A \leq N_n^{1 - \frac{c}{\log \log k}} \leq (QN)^{1 - \frac{c}{\log \log k}} \ll_k N^{1 - \frac{c}{\log \log k}}.\] The result follows.

\section{Entropy inequalities in positive characteristic}\label{fin-field-ent}

In this section we give the proof of Theorem \ref{thm13}. Suppose that $\X$ and $\Y$ are two $\F_p^{\infty}$-valued random variables, both taking finitely many values. Suppose that 
\begin{equation}\label{to-prove-3} \bH(\X - \Y) \geq (1 + \eps) \sup_{r \neq -1} \bH(\X + r\Y).\end{equation} Our aim is to prove that $\eps = O(\frac{1}{\log p})$, which immediately implies Theorem \ref{thm13}.

The initial phases of the argument mirror the deduction of Conjecture \ref{eak} from Conjecture \ref{ak-conj}. We may assume that there is some $q$ such that $q \P((\X, \Y) = (x, y)) \in \Z$ for all $(x,y)$; if \eqref{to-prove-3} can be established in this case, uniformly in $q$, then the general result follows by an easy approximation argument on letting $q \rightarrow \infty$.

Now let $m$ be very large, write $n = mq$, and construct a set $B^{(n)} \subset (\F_p^{\infty})^{qm} \times (\F_p^{\infty})^{qm}$ as follows. Let it consist of all pairs $((x_1,\dots, x_{mq})$, $(y_1,\dots, y_{mq}))$ for which 
\[ \# \{ i : (x_i, y_i) = (x,y)\} = mq \P((\X, \Y) = (x,y)).\] By arguments essentially the same as we saw before, 
\[ \bH(\X + r\Y) = \frac{1}{n} \log \pi_r(B^{(n)}) + o_{n \rightarrow \infty}(1).\] Hence, taking $m$ sufficiently large (and observing that $(\F_p^{\infty})^{qm}$ is isomorphic to $\F_p^{\infty}$ as a vector space), we obtain arbitrarily large sets $B \subset \F_p^{\infty} \times \F_p^{\infty}$ such that 
\begin{equation}\label{to-prove-4} \# \pi_{-1}(B)  \geq \sup_{r \neq -1} (\# \pi_{r}(B))^{1 + \eps/2}. \end{equation}
Note in particular that $\pi_{-1}(B)$ becomes arbitrarily large.

For such a $B$, we construct a finite set $A \subset \F_p^{\infty}$ as follows. If $(x,y) \in B$ and $x \neq y$, include the entire progression (line) through $x$ and $y$ in $A$. The points on this line are $\frac{x + ry}{1 + r}$, for $r \neq -1$, and $y$. Therefore
\[ A \subset \pi_{\infty}(B) \cup \bigcup_{r \neq -1} \frac{1}{1 + r} \cdot \pi_r(B),\] and therefore
\[ \# A \leq p \sup_{r \neq -1} \pi_r(B).\]
On the other hand, $A$ contains a progression of length $p$ (line) and common difference $d$, for every $d \in \pi_{-1}(B) \setminus \{0\}$. Thus, writing $N := \pi_{-1}(B) - 1$, we have
\begin{equation}\label{upper-5} \# A \ll_p N^{\frac{1}{1 + \eps/2}}. \end{equation}

On the other hand we have the following result, whose proof we will supply shortly.

\begin{proposition}\label{prop41}
Suppose that $A \subset \F_p^{\infty}$ is a finite set containing a progression of length $p$ \emph{(}that is, a line\emph{)} and common difference $d$, for all $d$ in some set of size $N$. Then $\# A \gg_p N^{1 - \frac{\log 2}{\log p} - o(1)}$.
\end{proposition}

Combining Proposition \ref{prop41} with the construction of $A$ satisfying \eqref{upper-5} immediately gives the desired upper bound $\eps = O(\frac{1}{\log p})$, thereby concluding the proof of Theorem \ref{thm13}.

It remains to prove Proposition \ref{prop41}.

\begin{proof}[Proof of Proposition \ref{prop41}] Set $A_1 := A$. In its initial stages, the proof of this result goes along rather similar lines to that of Proposition \ref{prop21}, only it is rather easier. The use of random projections in a similar context may be found in \cite[\S 3]{ellenberg-oberlin-tao}.
Let $n$ be the smallest positive integer for which $p^n \geq N$. 

Since $A_1$ is finite, it is contained in some copy of $\F_p^M$. Let $\pi : \F_p^M \rightarrow \F_p^n$ be a random linear map, selected by choosing the images of the basis vectors $e_1,\dots, e_M$ uniformly at random from $\F_p^n$. Set $A_2 := \pi(A_1)$; evidently $\# A_2 \leq \# A_1$. Let $\mathscr{D}$ be the set of common differences of progressions (of length $p$) lying in $A_1$. Then $A_2$ contains a progression of length $p$ and common difference $\pi(d)$, for every $d \in \mathscr{D}$.

Put some arbitrary order $\prec$ on $\mathscr{D}$, and suppose that $d \prec d'$. Then $\pi(d) = \pi(d')$ if and only if $\pi(d - d') = 0$. However, $\pi(d - d')$ is uniformly distributed in $\F_p^n$, and so the probability of this happening is $p^{-n}$. It follows that the expected number of pairs $(d, d')$ with $d \prec d'$ and $\pi(d) = \pi(d')$ is $p^{-n} \binom{N}{2} \leq \frac{1}{N} \binom{N}{2} \leq N/2$. Pick some map $\pi$ for which the number of such pairs is at most $N$. For each $v \in \F_p^n$, write $f(v) := \# \pi^{-1}(v)$. Then we have $\sum_v \binom{f(v)}{2} \leq N/2$, from which we obtain, since $\sum_v f(v) = N$, that $\sum_v f(v)^2 \leq 2N$. By Cauchy-Schwarz, 
\[ N^2 = \big(\sum_v f(v)\big)^2 \leq \# \{ v : f(v) \neq 0\} \sum_v f(v)^2,\] and therefore there are at least $N/2$ values of $v$ for which $f(v) \neq 0$. From the choice of $n$ it is clear that $p^n \leq pN$, and so at least $(\# \F_p^n)/2p$ elements of $\F_p^n$ lie in the image of $\pi$, or in other words are common differences of progressions in $B$.

By a random translation argument (see Corollary \ref{a3}), there is a set $A_3 \subset \F_p^n$, $\# A_3 \ll (n p \log p) \# A_2$, containing a line in every direction. That is, $A_3$ is a finite field Besicovitch set.

Now we bring in bounds on the size of such sets of a strength which, famously, are available in the finite field setting but not in characteristic zero. By the main result of \cite{dkss} we have $\# A_3 \geq (p/2)^n = (p^n)^{1 - \frac{\log 2}{\log p}} \geq N^{1 - \frac{\log 2}{\log p}}$. The proposition follows.
\end{proof}

\emph{Remarks.} Note that here it was crucial to have an effective lower bound on the size of Kakeya sets for fixed $p$ but with $n \rightarrow \infty$. For this, the celebrated work of Dvir \cite{dvir} on the Kakeya problem would not suffice. However (at the cost of weakening the exponents slightly) we could have used the main result of \cite{ss}, which has a slightly simpler proof than that of \cite{dkss}.

The $O(\frac{1}{\log p})$ term in Theorem \ref{ck-lower} is sharp. To see this, pick $a,b,b'$ independently and uniformly from $\F_p$, and  define random variables $\X, \Y$ taking values in $\F_p^2$ by
\[ \X = (a + b, ab), \quad \Y = (a + b', ab').\] Then
\[ \X - \Y = (b - b', a(b - b')),\]  which is almost uniformly distributed on $\F_p^2$: a short calculation gives
\[ \bH(\X - \Y) = 2 \log p + O(\frac{\log p}{p}).\] By contrast, if $r \neq -1$ then
\[ \frac{\X + r\Y}{1 + r} = \big(  a + \frac{b + rb'}{1 + r}, a \cdot \frac{b + rb'}{1 + r})  \big), \] and so $\X + r\Y$ is supported on a dilate of the set $V := \{(u + v, uv) : u,v \in \F_p\}$, which has cardinality $\frac{1}{2}p^2 + O(p)$. Therefore
\[ \bH(\X + r\Y) \leq 2 \log p - \log 2 + O(\frac{\log p}{p}).\]
Cognoscenti will recognise $V$ as being equivalent to the well-known construction of optimal Kakeya sets in $\F_p^2$, due to Mockenhaupt and Tao \cite{mockenhaupt-tao}.

\appendix

\section{Covering by translates}

In this section we review some standard lemmas on random translates.

\begin{lemma}\label{add-comp}
Suppose that $S \subset \{1,\dots, X\}$ is a set. Then there is a set $T$ of size $\ll \frac{X}{\# S} \log X$ such that $S + T \supset \{1,\dots, X\}$. 
\end{lemma}
\begin{proof}
We inductively define $t_1,t_2,\dots \in \{-X+1,\dots, X\}$ and $A_i := \{1,\dots, X\} \setminus \bigcup_{j = 1}^{i}(S + t_j)$ such that, given the choice of $t_1,\dots, t_i$, $\# A_{i+1}$ is as small as possible. We have
\[ \sum_t \# \big(A_i \cap (S + t)\big) = \# A_i \# S,\] and so \[ \max_t \# \big(A_i \cap (S + t)\big) \geq \# A_i \frac{\# S}{2X}.\]
Therefore
\[ \# A_{i+1} \leq \# A_i \big(1 - \frac{\# S}{2X}\big).\]
This process terminates with $\# A_i < 1$ (and hence $\# A_i = 0$) in $\ll \frac{X}{\# S}\log X$ steps.
\end{proof}

\begin{lemma}\label{lema2}
Suppose that $S \subset \F_p^n$ is a set. Then there is a set $T \subset \F_p^n$ of size $\ll \frac{p^n}{\# S} n \log p$ such that $S + T = \F_p^n$.
\end{lemma}
\begin{proof}
Very similar to the previous lemma, and left as an exercise.
\end{proof}

\begin{corollary}\label{a3}
Suppose that $A \subset \F_p^n$ is a set containing a $k$-term arithmetic progression with common difference $d$, for all $d$ lying in some set $\mathscr{D}$ of size $\delta p^n$. Then there is a set $A'$, $\# A' \ll_{k,n} \log p \cdot  \# A$, containing a $k$-term arithmetic progression with every common difference.
\end{corollary}
\begin{proof}
Apply Lemma \ref{lema2} with $S = \mathscr{D}$, and let $T$ be the resulting set. Then take $A' = \bigcup_{x \in \{0\} \cup T \cup \dots \cup (k-1) \cdot T} (A + x)$. 
\end{proof}

\end{document}